\documentclass[a4paper,11pt]{amsart}

\usepackage{hyperref} 
\usepackage[lite]{amsrefs} 
\usepackage{microtype} 
\usepackage{verbatim} 
\usepackage{color}  
\usepackage{amsmath,amssymb} 
\usepackage[all]{xy} 
\usepackage{enumerate} 
\usepackage{mathtools} 
\usepackage{mathrsfs}

\title[Profinite commensurability of \textit{S}-arithmetic groups]{Profinite commensurability\\ of \textit{S}-arithmetic groups}
\author{Holger Kammeyer}
\address{Institute for Algebra and Geometry\\ Karlsruhe Institute of Technology\\ Englerstr. 2 (Mathebau 20.30) \\ 76131 Karlsruhe \\ Germany}
\email{holger.kammeyer@kit.edu}
\urladdr{www.math.kit.edu/iag7/~kammeyer/}


\newtheorem{theorem}{Theorem}
\newtheorem{corollary}{Corollary}
\newtheorem{proposition}{Proposition}

\newtheorem{question}{Question}
\theoremstyle{definition}
\newtheorem{definition}{Definition}

\theoremstyle{remark}

\newtheorem{example}{Example}

\makeatletter
   \let\c@corollary=\c@theorem
   \let\c@proposition=\c@theorem
   \let\c@lemma=\c@theorem
   \let\c@definition=\c@theorem
   \let\c@remark=\c@theorem
   \let\c@example=\c@theorem
   \let\c@equation=\c@theorem
   \let\c@conjecture=\c@theorem
   \let\c@question=\c@theorem
\makeatother

\newcommand*{\MRref}[2]{ \href{http://www.ams.org/mathscinet-getitem?mr=#1}{MR \textbf{#1}}}
\newcommand*{\arXiv}[1]{ \href{http://www.arxiv.org/abs/#1}{arXiv:\textbf{#1}}}

\newcommand*{\Z}{\mathbb Z}
\newcommand*{\Q}{\mathbb Q}
\newcommand*{\R}{\mathbb R}
\newcommand*{\C}{\mathbb C}

\DeclareMathOperator{\rank}{rank}

\DeclareMathOperator{\lie}{Lie}

\DeclareMathOperator{\sgn}{sgn}
\DeclareMathOperator{\Gal}{Gal}
\DeclareMathOperator{\GL}{GL}

\newcounter{commentcounter}

\usepackage{ifthen,srcltx}
\newcommand{\showcomments}{yes}

\newsavebox{\commentbox}
\newenvironment{com}%
{\ifthenelse{\equal{\showcomments}{yes}}%
{\footnotemark
        \begin{lrbox}{\commentbox}
        \begin{minipage}[t]{1.25in}\raggedright\sffamily\tiny
        \footnotemark[\arabic{footnote}]}
{\begin{lrbox}{\commentbox}}}
{\ifthenelse{\equal{\showcomments}{yes}}
{\end{minipage}\end{lrbox}\marginpar{\usebox{\commentbox}}}
{\end{lrbox}}}

\DefineSimpleKey{bib}{myurl}

\newcommand\myurl[1]{\url{#1}}

\BibSpec{webpage}{%
  +{}{\PrintAuthors} {author}
  +{,}{ \textit} {title}
  +{,}{ } {note}
  +{,}{ \myurl} {myurl}
  +{}{ \parenthesize} {date}
  +{.}{ } {transition}
}

\newcommand{\ignore}[1]{}
\begin{document}

\begin{abstract} 
  Given an \(S\)-arithmetic group, we ask how much information on the ambient algebraic group, number field of definition, and set of places \(S\) is encoded in the commensurability class of the profinite completion.  As a first step, we show that the profinite commensurability class of a higher rank \(S\)-arithmetic group determines the number field up to arithmetical equivalence and the places in \(S\) above unramified primes.  We include applications to profiniteness questions of group invariants.
\end{abstract} 

\maketitle

\section{Introduction and setup}

Let \(k\) be a number field, let \(\mathbf{G}\) be a simply-connected absolutely almost simple linear algebraic \(k\)-group and let \(S\) be a finite set of places of \(k\) containing all the infinite ones.  Picking a \(k\)-embedding \(\mathbf{G} \subset \mathbf{GL_n}\), the three objects \(\mathbf{G}\), \(k\) and \(S\) define what we want to call a \emph{standard \(S\)-arithmetic group} \(\Gamma = \mathbf{G}(\mathcal{O}_{k, S}) = \mathbf{G}(k) \cap \mathbf{GL_n}(\mathcal{O}_{k, S})\), where \(\mathcal{O}_{k,S}\) are the \(S\)-integers of \(k\).  We say more generally that an abstract group \(\Gamma\) is \emph{\(S\)-arithmetic} if there are \(\mathbf{G}\), \(k\) and \(S\) as above such that \(\Gamma\) can be embedded as a subgroup of \(\mathbf{G}(\overline{k})\) which is commensurable to the standard \(S\)-arithmetic group \(\mathbf{G}(\mathcal{O}_{k, S})\) for one (then any) \(k\)-embedding \(\mathbf{G} \subset \mathbf{GL_n}\).  We call \((k, \mathbf{G}, S)\) an \emph{ambient triple} for \(\Gamma\).  Recall that the \emph{\(S\)-rank} of \(\mathbf{G}\) is defined by
  \[ \rank_S(\mathbf{G}) = \sum_{v \in S} \rank_{k_v} (\mathbf{G}) \]
  where \(\rank_{k_v}(\mathbf{G})\) is the dimension of a maximal \(k_v\)-split torus in \(\mathbf{G}\).
  \begin{definition}
    We say that \((k,\mathbf{G},S)\) has \emph{higher rank} if \(\rank_S(\mathbf{G}) \ge 2\) and if \(\rank_{k_v} (\mathbf{G}) \ge 1\) for all finite \(v \in S\) (``no compact \(p\)-adic factors'').
\end{definition}
  We say that \(\Gamma\) has \emph{higher rank} if it has an ambient triple of higher rank.  Two ambient triples \((k, \mathbf{G}, S)\) and \((l, \mathbf{H}, T)\) are called \emph{equivalent} if there is a field isomorphism \(\sigma \colon k \rightarrow l\) such that \(\prescript{\sigma}{}{\!S} = T\) and such that \(\prescript{\sigma}{}{\!\mathbf{G}}\) is \(l\)-isomorphic to \(\mathbf{H}\).  Here the set \(\prescript{\sigma}{}{\!S}\) consists of all places of \(l\) of the form \(w = v \circ \sigma^{-1}\) for \(v \in S\), and \(\prescript{\sigma}{}{\!\mathbf{G}}\) is the \(l\)-group obtained from \(\mathbf{G}\) via \(\sigma\).  An equivalence class \([k, \mathbf{G}, S]\) will be called an \emph{ambient class}.  Clearly, if one triple in \([k, \mathbf{G}, S]\) is ambient for \(\Gamma\), then so is any other triple in \([k, \mathbf{G}, S]\).  It is moreover well-defined to say that an ambient class \([k,\mathbf{G},S]\) has higher rank.  \emph{Margulis superrigidity} combined with \emph{strong approximation} implies that a higher rank \(S\)-arithmetic group \(\Gamma\) determines a higher rank ambient class \([k, \mathbf{G}, S]\) uniquely.  Details will be given in Proposition~\ref{prop:ambientclassunique}.  

  \subsection{Profinite commensurability} We call two groups (abstractly) \emph{commensurable} if they have isomorphic finite index subgroups.  The discussion so far says that ambient classes provide the commensurability classification of higher rank \(S\)-arithmetic groups: a commensurability class \([\Gamma]\) of higher rank \(S\)-arithmetic groups \(\Gamma\) determines an ambient class \([k, \mathbf{G}, S]\) and an ambient class \([k, \mathbf{G}, S]\) defines the commensurability class \([\mathbf{G}(\mathcal{O}_{k, S})]\).  These two constructions are well-defined and inverses of one another (Corollary~\ref{cor:commensurabilityclassification}).  The \emph{profinite completion} \(\widehat{\Gamma}\) of a group \(\Gamma\) is the projective limit over the inverse system of all finite quotients of \(\Gamma\).

  \begin{definition}
    Two groups are called \emph{profinitely commensurable} if their profinite completions have isomorphic open subgroups.
  \end{definition}

    For finitely generated groups, like our \(S\)-arithmetic groups, we could have equivalently required that the profinite completions have isomorphic finite index subgroups.  This is a consequence of a deep theorem due to Nikolov-Segal~\cite{Nikolov-Segal:strong-completeness}. It is readily verified that commensurable \(S\)-arithmetic groups are also profinitely commensurable (Proposition~\ref{prop:commimpliesprofcomm}).  Viewing that the \emph{commensurability class} of a higher rank \(S\)-arithmetic group \(\Gamma\) remembers the entire ambient information, it is a remarkable observation due to M.\,Aka~\cite{Aka:Arithmetic} that the \emph{profinite commensurability class} of \(\Gamma\) determines neither \(k\), nor \(\mathbf{G}\), nor \(S\).   Here are some examples.
    
  \begin{enumerate}[(i)]
  \item \label{item:arithequivfields} For \(k = \Q(\sqrt[8]{3})\) and \(l = \Q(\sqrt[8]{48})\) let \(S\) and \(T\) consist of all infinite places of \(k\) and \(l\), respectively, and of the finite places lying over \(2\) and~\(3\).  Then \(\mathbf{SL_3}(\mathcal{O}_{k, S})\) and \(\mathbf{SL_3}(\mathcal{O}_{l, T})\) are profinitely commensurable.
    
  \item \label{item:spin} For a real quadratic number field \(k\), the groups \(\textbf{Spin(6,1)}(\mathcal{O}_k)\) and \(\textbf{Spin(5, 2)}(\mathcal{O}_k)\) are profinitely commensurable.
    
  \item \label{item:galoisconj} Let \(k / \Q\) be Galois and let each of \(S\) and \(T\) contain precisely one pair of finite places, all four lying over a fixed split prime \(p\).  Then \(\mathbf{SL_3}(\mathcal{O}_{k, S})\) and \(\mathbf{SL_3}(\mathcal{O}_{k, T})\) are profinitely commensurable but not isomorphic unless the two pairs are \(\Gal(k/\Q)\)-conjugate.
  \end{enumerate}

  To appreciate these examples, be aware that \(S\)-arithmetic groups, as we defined them above, are finitely generated and linear, whence \emph{residually finite}: they embed densely into their profinite completion.  All these observations call for a systematic study of higher rank \(S\)-arithmetic groups from the profinite point of view.

  \begin{question} \label{question:what-notion}
    What are the strongest notions of equivalence between the fields \(k\) and \(l\), the groups \(\mathbf{G}\) and \(\mathbf{H}\), and the sets \(S\) and \(T\) that can be concluded from higher rank groups \(\mathbf{G}(\mathcal{O}_{k, S})\) and \(\mathbf{H}(\mathcal{O}_{l, T})\) being profinitely commensurable?
  \end{question}

  Ideally, requiring the three notions of equivalence simultaneously would recover the equivalence relation of profinite commensurability.  But there is no a priori reason why profinite commensurability would have such a tripartite description.  With Theorem~\ref{thm:profcomm-implies-arith-equiv} and Theorem~\ref{thm:profcomm-implies-similar-s} below, we offer first results for the fields and the places.

  \begin{theorem} \label{thm:profcomm-implies-arith-equiv}
    Suppose that \(\mathbf{G}(\mathcal{O}_{k,S})\) and \(\mathbf{H}(\mathcal{O}_{l,T})\) have higher rank and are profinitely commensurable.  Then \(k\) and \(l\) are arithmetically equivalent.
  \end{theorem}

  The shortest way to define \emph{arithmetical equivalence} of two number fields is to say that they have the same Dedekind zeta function.  An example is given by the pair \(k = \Q(\sqrt[8]{3})\) and \(l = \Q(\sqrt[8]{48})\) from \eqref{item:arithequivfields} above.  We discuss the notion of arithmetical equivalence more thoroughly in Section~\ref{subsection:arith-equiv-fields}.

    Theorem~\ref{thm:profcomm-implies-arith-equiv} says that the notion of equivalence in Question~\ref{question:what-notion} between \(k\) and \(l\) must at least be as strong as arithmetical equivalence.  But in fact, for any pair of arithmetically equivalent fields \(k\) and \(l\), a generalization of example~\eqref{item:arithequivfields} above gives \(S\)-arithmetic groups defined over \(k\) and \(l\) which are profinitely commensurable (Proposition~\ref{prop:arith-equiv-fields-give-profinitely-iso-groups}).  Hence the notion of equivalence between \(k\) and \(l\) cannot be stronger than arithmetical equivalence, either.
    
    We will see in Theorem~\ref{thm:consequences-arith-equiv}\,\eqref{item:same-discriminant} below that arithmetically equivalent fields have equal discriminant.  Therefore the same rational primes \(p\) ramify in \(k\) and \(l\).  Let \(S_p\) and \(T_p\) be the set of places in \(S\) and \(T\) lying over \(p\).
    
  \begin{theorem} \label{thm:profcomm-implies-similar-s}
    Suppose \(\mathbf{G}(\mathcal{O}_{k,S})\) and \(\mathbf{H}(\mathcal{O}_{l,T})\) have higher rank and are profinitely commensurable.  Then for each unramified prime \(p\), there is a residue degree preserving bijection \(S_p \rightarrow T_p\).
  \end{theorem}

Similarly as above, this notion of equivalence between \(S\) and \(T\) is close to optimal (Proposition~\ref{prop:specifiedbijection}).  It could only be improved by an additional assertion on the places over ramified primes.

\subsection{Applications to profiniteness of group invariants}  Once one has realized that residually finite groups with the same profinite completion might not be isomorphic, it becomes a fair question to ask what invariants of a group are actually already invariants of the profinite completion.  Similarly, one can ask what commensurability invariants are already profinite commensurability invariants.  These questions have attracted quite some research efforts recently and shall also be the theme of this section.

Aka's example~\eqref{item:spin} shows that the profinite commensurability class does not determine the group \(\mathbf{G}\) and hence neither does it determine the real Lie group \(G = \prod_{v \mid \infty} \mathbf{G}(k_v)\) in which an arithmetic group is a lattice.  But even if one fixes a \(\Q\)-group \(\mathbf{G}\), one can still construct profinitely commensurable \(S\)-arithmetic groups with different surrounding Lie group \(G_S = \prod_{v \in S} \mathbf{G}(k_v)\) by considering \(\mathbf{G}\) over different fields.  Here is an example.

  \begin{proposition} \label{prop:ramified-primes-in-s}
    Let \(k = \Q(\sqrt[8]{97})\) and \(l = \Q(\sqrt[8]{1552})\) and let \(S\) and \(T\) consist of all places lying over \(2\) and \(97\).  Then \(\mathbf{SL_3}(\mathcal{O}_{k, S})\) and \(\mathbf{SL_3}(\mathcal{O}_{l, T})\) have isomorphic profinite completions but \({(\textup{SL}_3)}_S\) and \({(\textup{SL}_3)}_T\) are not isomorphic.
    \end{proposition}

   The fields \(k\) and \(l\) are arithmetically equivalent and have common discriminant \(-2^{10} 97^7\) so that we inverted precisely the places over ramified primes.  To exclude this possibility, let us say that the set \(S\) is \emph{unramified} if it contains no places over ramified (finite) primes. For unramified \(S\), we finally obtain from Theorem~\ref{thm:profcomm-implies-arith-equiv} that the profinite commensurability class determines the Lie group \(G_S\).

    \begin{proposition} \label{prop:arith-group-determines-lie-group}
    For a fixed simply connected, absolutely almost simple linear algebraic \(\Q\)-group \(\mathbf{G}\), suppose that \([k, \mathbf{G}, S]\) and \([l, \mathbf{G}, T]\) are of higher rank and profinitely commensurable. Assume in addition that \(S\) is unramified.  Then the Lie group \(G_S\) is isomorphic to the Lie group \(G_T\).
  \end{proposition}
    
    Invoking recent work of Kyed--Petersen--Vaes~\cite{Kyed-Petersen-Vaes:LocallyCompact}, Proposition~\ref{prop:arith-group-determines-lie-group} implies the following profiniteness property for \emph{\(\ell^2\)-Betti numbers} \citelist{\cite{Kammeyer:IntroL2} \cite{Lueck:L2-Invariants}}.

    \begin{theorem} \label{thm:l2bettiprofinite}
      For a fixed simply connected, absolutely almost simple linear algebraic \(\Q\)-group \(\mathbf{G}\), suppose that \([k, \mathbf{G}, S]\) and \([l, \mathbf{G}, T]\) are of higher rank and profinitely commensurable. Assume in addition that \(S\) is unramified.  Then \(b^{(2)}_n(\mathbf{G}(\mathcal{O}_{k,S})) = 0\) if and only if \(b^{(2)}_n(\mathbf{G}(\mathcal{O}_{l,T})) = 0 \).
    \end{theorem}

    For instance, the theorem applies to the groups \(\Gamma = \mathbf{Spin(3,2)}(\mathcal{O}_{k, S})\) and \(\Lambda = \mathbf{Spin(3,2)}(\mathcal{O}_{l, T})\) with \(k\), \(l\), \(S\), and \(T\) as follows. We let \(k = \Q(\theta_1)\) and \(l= \Q(\theta_2)\) where \(\theta_1\) is a root of the first and \(\theta_2\) is a root of the second polynomial given in~\eqref{equation:mantilla} below.  For the finite sets \(S\) and \(T\), any choice is fine provided \(S_2 \cup S_{13} \cup S_{191}\) is empty and provided there are residue degree preserving bijections \(S_p \rightarrow T_p\) for all (finite) \(p\).  We will see on p.\,\pageref{page:localfields} that both \(\Gamma\) and \(\Lambda\) have a positive \(n\)-th \(\ell^2\)-Betti number if and only if \(n = 7 +2|S|\).

    In contrast, the groups \(\Gamma = \mathbf{Spin(6,1)}(\mathcal{O}_k)\) and \(\Lambda = \mathbf{Spin(5,2)}(\mathcal{O}_k)\) from example~\eqref{item:spin} are also profinitely commensurable but \(b^{(2)}_n(\Gamma) > 0\) if and only if \(n=6\) and \(b^{(2)}_n(\Lambda) > 0\) if and only if \(n = 10\).  So some restriction on the group \(\mathbf{G}\) will always be necessary when asserting profiniteness properties for \(\ell^2\)-Betti numbers of \(S\)-arithmetic groups.  We remark that Aka came up with the latter groups as examples of profinitely isomorphic groups with and without Kazhdan's property \((T)\)~\cite{Aka:PropertyT}.  A detailed account including the computation for \(\ell^2\)-Betti numbers can be found in the Master thesis of N.\,Stucki~\cite{Stucki:Master}.  Considering spinor groups over more sophisticated quadratic forms, one can actually show that the \(k\)-th \(\ell^2\)-Betti number is never a profinite invariant among \(S\)-arithmetic groups \cite{Kammeyer-Sauer:spinor} for \(k \ge 2\).  In contrast, the first \(\ell^2\)-Betti number is a profinite invariant even among finitely presented groups as a consequence of L\"uck approximation~\cite{Bridson-et-al:Fuchsian}.
    
    By Selberg's lemma, an \(S\)-arithmetic group \(\Gamma\) has a torsion-free subgroup \(\Gamma_0\) of finite index. In~\cite{Borel-Serre:Immeubles}*{Proposition~6.10}, Borel and Serre show that \(\Gamma_0\) acts freely and cocompactly on a product of a ``bordified'' symmetric space and certain Bruhat--Tits buildings.  The space is contractible and the quotient is triangulable so that it defines a finite classifying space \(B\Gamma_0\).  Thus the virtual Euler characteristic \(\chi(\Gamma) = \chi(B\Gamma_0)/[\Gamma : \Gamma_0]\) is defined (and well-defined).  Theorem~\ref{thm:l2bettiprofinite} has the consequence that the sign of the Euler characteristic is constant throughout the profinitely commensurable groups in question.

    \begin{corollary} \label{cor:signofeuler}
      For a fixed simply connected, absolutely almost simple linear algebraic \(\Q\)-group \(\mathbf{G}\), suppose that \([k, \mathbf{G}, S]\) and \([l, \mathbf{G}, T]\) are of higher rank and profinitely commensurable.  Assume in addition that \(S\) is unramified.  Then \(\sgn \chi(\mathbf{G}(\mathcal{O}_{k,S})) = \sgn \chi(\mathbf{G}(\mathcal{O}_{l,T}))\).
    \end{corollary}

    Here, as usual, the function \(\sgn(x)\) takes the values \(-1\), \(0\) or \(1\) according to whether \(x\) is \(<0\), \(=0\) or \(>0\).  We will explain at the end of the paper that it is possible to drop the assumption of unramified \(S\) and \(T\) both from Theorem~\ref{thm:l2bettiprofinite} and from Corollary~\ref{cor:signofeuler} if one assumes instead that the group \(\mathbf{G}\) splits over \(k\) and \(l\).  In the arithmetic case, when \(S\) consists of the infinite places only, a more in-depth analysis shows that the sign of the Euler characteristic is a profinite invariant even if the group \(\mathbf{G}\) is allowed to vary.  The absolute value of the Euler characteristic is then however not a profinite invariant.  These results will appear in~\cite{Kammeyer-et-al:profinite-invariants}.

    \subsection{Outline}  The article is organized as follows.  In Section~\ref{section:preliminaries} we provide necessary background material.  To wit, Section~\ref{subsection:arith-equiv-fields} is a quick report on arithmetically equivalent fields, with characterizations, properties and examples.  Section~\ref{subsection:strongapproximation} formulates strong approximation and defines the congruence completion in an adelic formulation.  Section~\ref{subsection:csp} surveys the main ideas of the congruence subgroup problem.  In Section~\ref{section:proofs}, the proofs of the results presented in this introduction are given.  The commensurability classification is given in Section~\ref{subsection:commensurability-classification}.  Section~\ref{subsection:consequences} proves the main results and Section~\ref{subsection:profiniteinvariants} concludes with the applications to profiniteness of group invariants.

    I wish to thank S.\,Kionke, G.\,Mantilla Soler, J.\,Raimbault, R.\,Sauer, and A.\,Valette for helpful communication.  I am also indebted to the Hausdorff Institute for Mathematics in Bonn for hosting the junior trimester program ``Topology'' in which part of this work was initiated.

\section{Preliminaries} \label{section:preliminaries}
 
In this section we collect necessary background material to prepare the proofs of the results from the introduction.

    \subsection{Arithmetically equivalent number fields} \label{subsection:arith-equiv-fields} Let \(k\) be a number field.  The \emph{decomposition type} of a rational prime \(p\) in \(k\) is the tuple
  \[ A_k(p) = (f_1, \ldots, f_r ) \]
  consisting of the residue degrees \(f_i = f(\mathfrak{p}_i|p)\) of all primes \(\mathfrak{p}_i\) lying over \(p\) in ascending order: \(f_1 \le \cdots \le f_r\).

  \begin{definition}
    Two number fields \(k\) and \(l\) are called \emph{arithmetically equivalent} if \(A_k(p) = A_l(p)\) for all  but finitely many rational primes \(p\).
  \end{definition}

  We will next see that this is only one out of several characterizations of this notion.  Let \(N/\Q\) be any Galois extension containing both \(k\) and \(l\) and let \(U\) and \(V\) be the subgroups of \(G = \Gal(N/\Q)\) corresponding to \(k\) and \(l\), respectively.  We denote by \(1^G_U\) and \(1^G_V\) the characters of the coset permutation representations \(\C[G/U]\) and \(\C[G/V]\).  Equivalently, \(1^G_U\) and \(1^G_V\) are the characters induced from the trivial (one-dimensional) characters \(1_U\) and \(1_V\).  Finally, let \(\zeta_k\) and \(\zeta_l\) be the Dedekind zeta functions of \(k\) and \(l\).

  \begin{theorem} \label{thm:characterizations-arith-equiv}
    The following are equivalent.
    \begin{enumerate}[(i)]
    \item The fields \(k\) and \(l\) are arithmetically equivalent. \label{item:definition-arith-equiv}
    \item We have \(A_k(p) = A_l(p)\) for \emph{all} rational primes \(p\). \label{item:all-decomp-types-equal}
    \item We have \(1^G_U = 1^G_V\). \label{item:characters-equal}
    \item We have \(\zeta_k = \zeta_l\).
    \end{enumerate}   
  \end{theorem}

  The result can be found in \cite{Klingen:ArithmeticalSimilarities}*{Theorem~III.1.3, p.\,77}.  Note that statement \eqref{item:characters-equal} gives an entirely group theoretic characterization of arithmetical equivalence which will be useful in a moment.  We include the proof of the corresponding equivalence \eqref{item:definition-arith-equiv} \(\Leftrightarrow\) \eqref{item:characters-equal} as it is brief and instructive.

  \begin{proof}
    \eqref{item:definition-arith-equiv} \(\Rightarrow\) \eqref{item:characters-equal}.  Given an unramified prime \(\mathfrak{Q}\) of \(N\) over a rational prime \(p\), reduction mod \(\mathfrak{Q}\) defines an isomorphism from the \emph{decomposition group} \(D(\mathfrak{Q}) = \{ \sigma \in G \colon \mathfrak{Q}^\sigma = \mathfrak{Q} \}\) to the Galois group \(\Gal(\mathcal{O}_N/\mathfrak{Q} \ /\  \Z / p)\) of the residue field.  The latter group is cyclic, canonically generated by the Frobenius automorphism \(x \mapsto x^p\).  The unique preimage \(F(\mathfrak{Q}) \in D(\mathfrak{Q})\) is called the \emph{Frobenius element} of \(\mathfrak{Q}\).  By Chebotarev's density theorem, each element \(g \in G\) is a Frobenius element \(g = F(\mathfrak{Q})\) for infinitely many primes \(\mathfrak{Q}\) in \(N\).  Thus we can choose \(\mathfrak{Q}\) above some unramified \(p\) with \(A_k(p) = A_l(p)\).  The element \(F(\mathfrak{Q})\) permutes the set \(G/U\) and the cycle lengths of this permutation, in ascending order, coincide with the decomposition type \(A_k(p)\) of \(p\) in \(k\) \cite{Klingen:ArithmeticalSimilarities}*{Theorem I.1.12, p.\,11}.  Hence \(1^G_U(g)\) is the number of primes \(\mathfrak{p_i}\) in \(k\) over \(p\) with residue degree \(f_i = 1\).  Since \(A_k(p) = A_l(p)\), it follows that \(1^G_U(g) = 1^G_V(g)\).
    
\eqref{item:characters-equal} \(\Rightarrow\) \eqref{item:definition-arith-equiv}.  The above shows more generally that the residue degrees \(f_i\) in \(k\) above any unramified prime \(p\) satisfy the relation \(1^G_U(g^s) = \sum_{f_i \mid s} f_i\) for any positive integer \(s\).  Thus all residue degrees above unramified primes can be computed recursively from the numbers \(1^G_U(g^s)\).
  \end{proof}

  Many field invariants turn out to be invariant under arithmetical equivalence because they can be calculated from the character \(1^G_U\) alone.  Here are some examples \cite{Klingen:ArithmeticalSimilarities}*{Theorem~III.1.4}.
  
  \begin{theorem} \label{thm:consequences-arith-equiv}
    The following invariants coincide for arithmetically equivalent number fields:
    \begin{enumerate}[(i)]
    \item degree over \(\Q\), \label{item:same-degree}
    \item Galois closure, \label{item:same-galois-closure}
    \item discriminant, \label{item:same-discriminant}
    \item signature. \label{item:same-signatures}
    \end{enumerate}
  \end{theorem}

  Recall that the \emph{signature} of a number field is the pair \((r,s)\) where \(r\) and \(s\) is the number of real and complex places.
  
  \begin{proof}
    Retaining the setting and notation from Theorem~\ref{thm:characterizations-arith-equiv}, we obtain \([k : \Q] = [G : U] = 1^G_U(1)\) which shows~\eqref{item:same-degree}.  The Galois closure of \(k\) corresponds to the normal core of \(U\) which is precisely the kernel of the permutation representation \(G \rightarrow \GL(\C[G/U])\).  Evidently, this kernel is the set of all \(\sigma \in G\) satisfying \(1^G_U(\sigma) = 1^G_U(1)\) which proves~\eqref{item:same-galois-closure}.  Artin's conductor-discriminant formula says the discriminant of \(k\) is given by the Artin conductor of \(1^G_U\) \cite{Neukirch:Zahlentheorie}*{Korollar~VII.11.8, p.\,557}.  This demonstrates~\eqref{item:same-discriminant}.  Complex conjugation defines an element \(c \in G\).  The value \(1^G_U(c)\) is the number of cosets \(gU\) fixed by \(c\).  The identity \(cgU = gU\) says precisely that \(cg\) and \(g\) restrict to the same transformations on \(k\); equivalently \(g(k) \subset \R\).  Thus \(1^G_U(c) = r\) is the number of real embeddings, and \(s = (1^G_U(1)-r)/2\) is the number of complex embeddings of \(k\) up to conjugation.  This shows~\eqref{item:same-signatures}.
  \end{proof}

  \begin{example} \label{example:arith-equiv-fields}
    We conclude this excursion on arithmetically equivalent fields with a brief report on examples of low degree.  According to~\cite{Perlis:OnTheEquation}, number fields of degree~\(\le 6\) over \(\Q\) are \emph{arithmetically solitary}, meaning any pair of arithmetically equivalent fields is already conjugate over \(\Q\).  In degree seven, the fields generated by a root of the polynomials
    \[ X^7 - 7X + 3 \quad \textup{and} \quad X^7 +14X^4 -42X^2 -21X +9,\] 
respectively, are examples of arithmetically equivalent fields which are not isomorphic~\cite{Trinks:Arithmetisch}.  More generally, one can construct infinitely many pairs of irreducible degree seven polynomials which define arithmetically equivalent but non-isomorphic number fields, including totally real ones, many of them with different class numbers~\cite{Bosma-deSmit:OnArithEquiv}.  I am grateful to Guillermo Mantilla for making me aware of the explicit polynomials
    \begin{align} \label{equation:mantilla}
      & X^7 -18X^5 -28X^4 +10X^3 +24X^2 -2 \quad \text{and} \\ & X^7-3X^6-15X^5+51X^4-19X^3-41X^2+13X+11 \nonumber
    \end{align}
    which, as we will see on p.\,\pageref{page:localfields}, define a pair of non-isomorphic totally real number fields \(k\) and \(l\) that are even \emph{adelically equivalent}, meaning the adele rings \(\mathbb{A}_k\) and \(\mathbb{A}_l\) are isomorphic (as \(\mathbb{A}_{\Q}\)-algebras or, equivalently, as topological rings).  Equivalently, there is a bijection \(\varphi\) between the finite places of \(k\) and \(l\) such that \(k_v \cong_{\Q_p} l_{\varphi(v)}\) for all \(v \mid p\) and all \(p\) \cite{Klingen:ArithmeticalSimilarities}*{Theorem~2.3.(b), p.\,237}.  Since \([k_v : \Q_p] = e_v f_v\) is the product of ramification index and residue degree, this shows that adelic equivalence is a stronger equivalence relation than arithmetical equivalence.
    
  In degree eight, examples of arithmetically and adelically equivalent fields occur among simple radical extensions.  If \(a\) is any square-free integer with \(|a| \ge 3\), then the fields \(k = \Q(\sqrt[8]{a})\) and \(l = \Q(\sqrt{2}\sqrt[8]{a})\) are arithmetically equivalent and not conjugate~\cite{Komatsu:OnTheAdeleRings}.  If (and only if) in addition \(a = 2, 7, 14, 15 \mod 16\), the fields \(k\) and \(l\) are adelically equivalent.  The eighth root \(\sqrt[8]{a}\) can be replaced by any \(2^r\)-th root \(\sqrt[2^r]{a}\) for \(r \ge 3\), creating examples of arithmetically and adelically equivalent fields of arbitrarily large degree.
  \end{example}

  \subsection{Adeles, strong approximation, and the congruence completion} \label{subsection:strongapproximation} Let \(\mathbb{A}_{k,S}\) be the ring of \emph{{\(S\)-truncated} adeles} of \(k\): the \emph{restricted product} \(\prod^r_{v \notin S} k_v\) consisting of all elements in \(\prod_{v \notin S} k_v\) with almost all coordinates in the valuation ring \(\mathcal{O}_v\) of \(k_v\).  The ring \(\mathbb{A}_{k, S}\) becomes a locally compact topological ring if we declare that rectangular sets \(\prod_{v \in E \setminus S} U_v \times \prod_{v \notin E} \mathcal{O}_v\) for finite \(E \supset S\) and open \(U_v \subset k_v\) form a basis of the topology.  We have a diagonal embedding \(k \rightarrow \mathbb{A}_{k,S}\) which is dense according to the \emph{strong approximation} theorem \cite{Platonov-Rapinchuk:AlgebraicGroups}*{Theorem~1.5, p.\,14}.  Here it was important that \(S\) is always nonempty as it contains the infinite places of \(k\).  The subspace topology of the open and closed subset \(\prod_{v \notin S} \mathcal{O}_v \subset \mathbb{A}_{k,S}\) is generated by the open sets \(\prod_{v \in E \setminus S} U_v \cap \mathcal{O}_v \times \prod_{v \notin E} \mathcal{O}_v\) so that strong approximation implies that the closure of \(\mathcal{O}_{k,S}\) under the embedding \(k \rightarrow \mathbb{A}_{k,S}\) is the product \(\prod_{v \notin S} \mathcal{O}_v\).  The latter has the equivalent algebraic description \(\prod_{v \notin S} \mathcal{O}_v = \widehat{\mathcal{O}_{k,S}}\), where \(\widehat{\mathcal{O}_{k,S}}\) is the \emph{profinite completion} of the ring \(\mathcal{O}_{k,S}\) defined as 
  \[ \widehat{\mathcal{O}_{k,S}} = \varprojlim\nolimits_{\mathfrak{a} \neq 0} \mathcal{O}_{k,S} / \mathfrak{a}, \]
  the projective limit along the inverse system of all factor rings by nonzero ideals \(\mathfrak{a} \subset \mathcal{O}_{k,S}\).

  To see this, we observe that the canonical map \(\mathcal{O}_{k,S} \rightarrow \widehat{\mathcal{O}_{k,S}}\) is an embedding.  The embedding is moreover dense by strong approximation, which in this setting is just a reformulation of the Chinese remainder theorem.  The nonzero ideals \(\mathfrak{a} \subset \mathcal{O}_{k,S}\) form a fundamental system of neighborhoods of zero for the subspace topology.  This topology agrees with the topology of \(\mathcal{O}_{k,S}\) as a subset of \(\prod_{v \notin S} \mathcal{O}_v\) because the sets \(\prod_{v \in E \setminus S} \mathfrak{p}_v^{n_v} \times \prod_{v \notin E} \mathcal{O}_v\) are such a fundamental system in \(\prod_{v \notin S} \mathcal{O}_v\), where \(n_v \ge 0\) and \(\mathfrak{p}_v \subset \mathcal{O}_v\) is the maximal ideal.  Thus both \(\widehat{\mathcal{O}_{k,S}}\) and \(\prod_{v \notin S} \mathcal{O}_v\) are completions of the topological ring \(\mathcal{O}_{k,S}\) (with respect to the canonical uniform structure) and therefore the identity map on \(\mathcal{O}_{k,S}\) extends uniquely to an isomorphism \(\widehat{\mathcal{O}_{k,S}} \cong \prod_{v \notin S} \mathcal{O}_v\).

  Similarly, for a \(k\)-subgroup \(\mathbf{G} \subset \mathbf{GL_n}\), we define the \(\mathbb{A}_{k,S}\)-points \(\mathbf{G}(\mathbb{A}_{k,S})\) as the restricted product \(\prod^r_{v \notin S} \mathbf{G}(k_v)\) where almost all coordinates lie in \(\mathbf{G}(\mathcal{O}_v)\).  We thus have a diagonal embedding of the \(k\)-rational points \(\mathbf{G}(k) \subset \mathbf{G}(\mathbb{A}_{k,S})\) and \(\mathbf{G}\) is said to have the \emph{strong approximation} property (with respect to \(S\)) if this embedding is dense.  Arguing as above, we see that strong approximation implies that \(\mathbf{G}(\mathcal{O}_{k,S})\) has closure \(\prod_{v \notin S} \mathbf{G}(\mathcal{O}_v)\) under this embedding.  The strong approximation theorem due to Kneser--Platonov \cite{Platonov-Rapinchuk:AlgebraicGroups}*{Theorem~7.12, p.\,427} asserts that for a \(k\)-simple, simply-connected group \(\mathbf{G}\), strong approximation with respect to \(S\) is equivalent to \(G_S\) being noncompact.  For higher rank groups, this condition is granted by the inequality \(\rank_S \mathbf{G} \ge 2\).  The embedding \(\mathbf{G}(\mathcal{O}_{k,S}) \subset \prod_{v \notin S} \mathbf{G}(\mathcal{O}_v)\) endows \(\mathbf{G}(\mathcal{O}_{k,S})\) with a subspace topology.  In this topology, a fundamental system of neighborhoods of the unit element \(e \in \mathbf{G}(\mathcal{O}_{k,S})\) consists of matrices whose entries differ from the entries of the unit matrix by elements in \(\mathcal{O}_{k,S}\) with certain prescribed minimal \(v\)-adic valuations for finitely many \(v\).  Algebraically, these neighborhoods can thus be described as the kernel \(\Gamma(\mathfrak{a})\) of the homomorphism \(\mathbf{G}(\mathcal{O}_{k,S}) \rightarrow \mathbf{G}(\mathcal{O}_{k,S} / \mathfrak{a})\) for the nonzero ideal \(\mathfrak{a} \subset \mathcal{O}_{k,S}\) given by the product of the prime ideal powers corresponding to the minimal \(v\)-adic valuations.  The groups \(\Gamma(\mathfrak{a})\) are also known as the \emph{principal congruence subgroups} of \(\mathbf{G}(\mathcal{O}_{k,S})\) where a \emph{congruence subgroup} would be any subgroup that contains some \(\Gamma(\mathfrak{a})\).  Arguing as above, we see that the identity map on \(\mathbf{G}(\mathcal{O}_{k,S})\) extends to a canonical isomorphism of topological groups
  \[ \overline{\mathbf{G}(\mathcal{O}_{k,S})} \cong \prod_{v \notin S} \mathbf{G}(\mathcal{O}_v) \]
  where \(\overline{\mathbf{G}(\mathcal{O}_{k,S})}\) is given by the projective limit \(\varprojlim_{\mathfrak{a} \neq 0} \mathbf{G}(\mathcal{O}_{k,S} / \mathfrak{a})\).  Consequently, \(\overline{\mathbf{G}(\mathcal{O}_{k,S})}\) is called the \emph{congruence completion} of \(\mathbf{G}(\mathcal{O}_{k,S})\).  Note that from the functorial view point, \(\mathbf{G}\) preserves products.  One can therefore understand the above isomorphism more suggestively as saying that \(\mathbf{G}\) is ``continuous at the zero ideal'':
  \[ \varprojlim\nolimits_{\mathfrak{a} \neq 0} \mathbf{G}\left(\mathcal{O}_{k,S} / \mathfrak{a}\right) = \mathbf{G}\left(\varprojlim\nolimits_{\mathfrak{a} \neq 0} \mathcal{O}_{k,S} / \mathfrak{a}\right). \]

  \subsection{The congruence subgroup problem} \label{subsection:csp} The congruence subgroups show that an \(S\)-arithmetic group \(\Gamma = \mathbf{G}(\mathcal{O}_{k,S})\) comes with a wealth of finite index subgroups.  This raises the question whether they actually provide for all (or essentially all) finite index subgroups of \(\Gamma\).  To state the precise, modern version of the problem, recall that \(\widehat{\Gamma}\) denotes the \emph{profinite completion} of \(\Gamma\), the projective limit along the inverse system of all finite quotient groups of \(\Gamma\).  Since the congruence completion \(\overline{\Gamma}\) is a profinite group as well, the universal property of \(\widehat{\Gamma}\) yields a canonical homomorphism \(\pi \colon \widehat{\Gamma} \rightarrow \overline{\Gamma}\) that restricts to the identity on \(\Gamma\).  It is always an epimorphism because the image is dense, as it contains \(\Gamma\), and closed because \(\widehat{\Gamma}\) is compact and \(\overline{\Gamma}\) is Hausdorff.  The kernel \(C(k, \mathbf{G}, S)\) of \(\pi\) is known as the \emph{congruence kernel} of the \(k\)-group \(\mathbf{G}\) with respect to~\(S\).  We thus have a short exact sequence
  \[ 1 \longrightarrow C(k, \mathbf{G}, S) \longrightarrow \widehat{\Gamma} \xrightarrow{ \ \pi \ } \overline{\Gamma} \longrightarrow 1. \]
  The \emph{congruence subgroup problem} asks to determine \(C(k, \mathbf{G}, S)\), the most important question being whether the profinite group \(C(k,\mathbf{G},S)\) is actually \emph{finite} in which case we say that the \(k\)-group \(\mathbf{G}\) has the \emph{congruence subgroup property} (CSP) with respect to \(S\).

  For an absolutely almost simple simply-connected \(k\)-group \(\mathbf{G}\), a conjecture of Serre~\cite{Platonov-Rapinchuk:AlgebraicGroups}*{(9.45), p.\,556} asserts that \(C(k, \mathbf{G}, S)\) is finite if \(\rank_S \mathbf{G} \ge 2\) and \(\rank_{k_v} \mathbf{G} \ge 1\) for all finite places \(v \in S\) and is infinite if \(\rank_S \mathbf{G} = 1\).  So assuming \((k, \mathbf{G},S)\) has higher rank should precisely ensure that \(\mathbf{G}\) has CSP.  The conjecture is known for the bigger part of all possible groups~\(\mathbf{G}\), including all \(k\)-isotropic ones.  The most notorious open case occurs when \(\mathbf{G}\) is an anisotropic inner form of type \(A_n\).

 For higher rank \((k,\mathbf{G}, S)\), it is expected more precisely that \(C(k, \mathbf{G}, S)\) should be isomorphic to the group of roots of unity \(\mu_k \subset k^*\) if \(k\) is totally imaginary and should be trivial if \(k\) is not.  Raghunathan informs us in~\cite{Raghunathan:CSP}*{p.\,304} that the proof of this more precise picture is meanwhile complete for \(k\)-isotropic \(\mathbf{G}\).  So at least if \(k\) has a real place and \(\rank_k \mathbf{G} \ge 1\), we are in the convenient situation that every finite index subgroup of \(\Gamma = \mathbf{G}(\mathcal{O}_{k,S})\) is a congruence subgroup.  In this case, strong approximation and CSP combine to the elegant canonical isomorphism
  \[ \widehat{\mathbf{G}(\mathcal{O}_{k,S})} \cong \mathbf{G}(\widehat{\mathcal{O}_{k,S}}). \]
  For more information on CSP, we refer the reader to the survey article~\cite{Prasad-Rapinchuk:Developments}.
  
  \section{Proofs} \label{section:proofs}

  In this section we prove the results that were outlined in the introduction.  We start by recalling how superrigidity and strong approximation imply the commensurability classification of higher rank \(S\)-arithmetic groups in Section~\ref{subsection:commensurability-classification}.  In Section~\ref{subsection:consequences}, we show how profinite commensurability of two such groups implies similarities of the corresponding \(k\) and \(S\) by proving Theorems~\ref{thm:profcomm-implies-arith-equiv} and~\ref{thm:profcomm-implies-similar-s}, together with the accompanying Propositions~\ref{prop:arith-equiv-fields-give-profinitely-iso-groups} and~\ref{prop:specifiedbijection} which give a partial converse.  Section~\ref{subsection:profiniteinvariants} concludes with the proofs of the applications to profiniteness questions of group invariants.

  \subsection{Commensurability classification}  \label{subsection:commensurability-classification} We show that ambient classes classify commensurability classes of higher rank \(S\)-arithmetic groups.
  
\begin{proposition} \label{prop:ambientclassunique}
  Any two higher rank ambient triples \((k, \mathbf{G}, S)\) and \((l, \mathbf{H}, T)\) of a higher rank \(S\)-arithmetic group \(\Gamma\) are equivalent.
\end{proposition}

  \begin{proof}
    Fix an embedding \(\Gamma \le \mathbf{G}(\overline{k})\) commensurable with \(\mathbf{G}(\mathcal{O}_{k, S})\) for some \(k\)-embedding \(\mathbf{G} \subset \mathbf{GL_n}\).  Since \((l, \mathbf{H}, T)\) is an ambient triple as well, we have another embedding \(\delta \colon \Gamma \rightarrow \mathbf{H}(\overline{l})\) commensurable with \(\mathbf{H}(\mathcal{O}_{l, T})\) for some \(k\)-embedding \(\mathbf{H} \subset \mathbf{GL_m}\).  Thus there is a finite index subgroup \(\Gamma_0 \le \Gamma \cap \mathbf{G}(\mathcal{O}_{k, S})\) which is mapped by \(\delta\) to a finite index subgroup in \(\mathbf{H}(\mathcal{O}_{l,T})\).  The group \(\delta(\Gamma_0)\) is Zariski dense in \(\mathbf{H}\) according to \cite{Margulis:DiscreteSubgroups}*{Proposition~I.3.2.10, p.\,64}.  By Margulis superrigidity \cite{Margulis:DiscreteSubgroups}*{Theorem~(C), p.\,259}, there is a homomorphism \(\sigma \colon k \rightarrow l\), an \(l\)-epimorphism \(\eta \colon \!\!\prescript{\sigma}{}{\!\mathbf{G}} \rightarrow \mathbf{H}\), and a homomorphism \(\nu \colon \Gamma_0 \rightarrow \mathcal{Z}(\mathbf{H})\) to the (finite) center of \(\mathbf{H}\) such that \(\delta(\gamma) = \nu(\gamma) \cdot \eta(\sigma^0(\gamma))\) for all \(\gamma \in \Gamma_0\).  Here \(\sigma^0 \colon \mathbf{G}(k) \rightarrow \!\!\prescript{\sigma}{}{\!\mathbf{G}(l)}\) is the group homomorphism induced by \(\sigma\).  The subgroup \(\Gamma_1 = \ker \nu\) still has finite index in \(\Gamma\), and we have \(\delta(\gamma) = \eta(\sigma^0(\gamma))\) for all \(\gamma \in \Gamma_1\).  As we can swap the roles of \(\Gamma\) and \(\Lambda\), there is also a homomorphism \(l \rightarrow k\), showing that \(\sigma\) is a field isomorphism.  Moreover, since \(\mathbf{H}\) is simply-connected, the central \(l\)-isogeny \(\eta\) is in fact an \(l\)-isomorphism.  

    It remains to show that \(\prescript{\sigma}{}{\!S} = T\).  To this end, note that according to \cite{Margulis:DiscreteSubgroups}*{Lemma~I.3.1.1.(iv), p.\,60}, the group \(\eta \circ \sigma^0(\mathbf{G}(\mathcal{O}_{k, S})) = \eta(\prescript{\sigma}{}{\!\mathbf{G}}(\mathcal{O}_{l, \prescript{\sigma}{}{\!S}}))\) is commensurable with \(\mathbf{H}(\mathcal{O}_{l, \prescript{\sigma}{}{\!S}})\).  As the homomorphism \(\eta \circ \sigma^0\) restricts to \(\delta\) on \(\Gamma_1\), we have that \(\mathbf{H}(\mathcal{O}_{l, \prescript{\sigma}{}{\!S}})\) is commensurable with \(\mathbf{H}(\mathcal{O}_{l, T})\).  This says that the intersection \(\mathbf{H}(\mathcal{O}_{l, \prescript{\sigma}{}{\!S}}) \cap \mathbf{H}(\mathcal{O}_{l, T}) = \mathbf{H}(\mathcal{O}_{l, \prescript{\sigma}{}{\!S} \cap T})\) has finite index both in \(\mathbf{H}(\mathcal{O}_{l, \prescript{\sigma}{}{\!S}})\) and in \(\mathbf{H}(\mathcal{O}_{l, T})\).  Therefore the proof is complete once we show that for general \(S_1 \subseteq S_2\) satisfying our assumptions, the group \(\mathbf{H}(\mathcal{O}_{l, S_1})\) has infinite index in \(\mathbf{H}(\mathcal{O}_{l, S_2})\) unless \(S_1 = S_2\).  So suppose \(S_2 \setminus S_1 \neq \emptyset\).  Then \(H_{S_2 \setminus S_1} = \prod_{v \in S_2 \setminus S_1} \mathbf{H}(l_v)\) is a nontrivial product of noncompact \(p\)-adic Lie groups because we assume that \(\mathbf{H}\) is \(l_v\)-isotropic for all finite places \(v \in S_2\).  The assumption \(\rank_{S_1} \mathbf{H} \ge 2\) says in particular that \(H_{S_1}\) is noncompact which, as discussed in Section~\ref{subsection:strongapproximation}, is equivalent to \(\mathbf{H}\) having the strong approximation property with respect to \(S_1\).  The density of \(\mathbf{H}(l)\) in \(\mathbf{H}(\mathbb{A}_{l, S_1})\) clearly implies that the diagonal embedding of \(\mathbf{H}(\mathcal{O}_{l, S_2})\) into \(H_{S_2 \setminus S_1}\) is dense.  On the other hand, \(\mathbf{H}(\mathcal{O}_{l, S_1})\) lies in the subgroup \(K_{S_2 \setminus S_1} = \prod_{v \in S_2 \setminus S_1} \mathbf{H}(\mathcal{O}_v)\) of \(H_{S_2 \setminus S_1}\) which is compact hence closed because \(H_{S_2 \setminus S_1}\) is Hausdorff.  Since a compact subgroup of a noncompact group has infinite index, we conclude
    \[ [\mathbf{H}(\mathcal{O}_{l, S_2}) : \mathbf{H}(\mathcal{O}_{l, S_1})] \ge [H_{S_2 \setminus S_1} : \overline{\mathbf{H}(\mathcal{O}_{l, S_1})}] \ge [H_{S_2 \setminus S_1} : K_{S_2 \setminus S_1}] = \infty. \qedhere \]
  \end{proof}
  
    If the embeddings \(\Gamma \rightarrow \mathbf{G}(\overline{k})\) and \(\Gamma \rightarrow \mathbf{H}(\overline{l})\) are specified, then Margulis superrigidity says additionally that the isomorphism \(\sigma \colon k \rightarrow l\) and the \(l\)-isomorphism \(\eta \colon \!\!\prescript{\sigma}{}{\!\mathbf{G}} \rightarrow \mathbf{H}\) exhibiting the equivalence of ambient triples are unique with the property that the diagram
      \[
      \begin{xy}
        \xymatrix{
          \mathbf{G}(k) \ar[r]^{\sigma^0}_\cong & \prescript{\sigma}{}{\!\mathbf{G}(l)} \ar[d]^\eta_\cong \\
          \Gamma_1 \ar@{^{(}->}[u] \ar@{^{(}->}[r] & \mathbf{H}(l)
        }
      \end{xy}
      \]
      commutes for a finite index subgroup \(\Gamma_1 \le \Gamma\).  So it is fair to say that the ambient triple of a higher rank \(S\)-arithmetic group is unique up to unique equivalence.  The above proposition shows in fact that ambient classes classify higher rank \(S\)-arithmetic groups up to abstract commensurability.

      \begin{corollary} \label{cor:commensurabilityclassification}
        Let \(\Gamma\) and \(\Lambda\) be higher rank \(S\)-arithmetic groups with higher rank ambient triples \((k, \mathbf{G}, S)\) and \((l, \mathbf{H}, T)\).  Then \(\Gamma\) and \(\Lambda\) are abstractly commensurable if and only if \([k, \mathbf{G}, S] = [l, \mathbf{H}, T]\).
      \end{corollary}

      \begin{proof}
        If \([k, \mathbf{G}, S] = [l, \mathbf{H}, T]\), then there are isomorphisms \(\sigma \colon k \rightarrow l\) and \(\eta \colon \!\!\prescript{\sigma}{}{\!\mathbf{G}} \rightarrow \mathbf{H}\) such that \(\prescript{\sigma}{}{\!S} = T\).  Thus \(\Gamma\) is commensurable with \(\mathbf{G}(\mathcal{O}_{k, S})\) which is isomorphic to \(\eta(\prescript{\sigma}{}{\!\mathbf{G}(\mathcal{O}_{l, T})})\) which is commensurable with \(\mathbf{H}(\mathcal{O}_{l, T})\) which is commensurable with \(\Lambda\).  Thus \(\Gamma\) is abstractly commensurable with \(\Lambda\).

          Conversely, let \(\delta \colon \Gamma_1 \rightarrow \Lambda_1\) be an isomorphism where \([\Gamma : \Gamma_1] < \infty\) and \([\Lambda : \Lambda_1] < \infty\).  Then \(\delta\) yields an embedding of \(\Gamma_1\) into \(\mathbf{H}(\overline{l})\) commensurable with \(\mathbf{H}(\mathcal{O}_{l, T})\).  From Proposition~\ref{prop:ambientclassunique} we conclude \([k, \mathbf{G}, S] = [l, \mathbf{H}, T]\).
      \end{proof}

      Finally, we show that commensurability is a stronger equivalence relation than profinite commensurability  for the groups of interest.

      \begin{proposition} \label{prop:commimpliesprofcomm}
        Let \(\Gamma\) and \(\Lambda\) be commensurable residually finite groups.  Then \(\Gamma\) and \(\Lambda\) are profinitely commensurable.
      \end{proposition}

      \begin{proof}
        Let \(H\) be a group which has finite index embeddings into \(\Gamma\) and \(\Lambda\).  Then the closures \(\overline{H}^{\widehat{\Gamma}}\) and \(\overline{H}^{\widehat{\Lambda}}\) are open in \(\widehat{\Gamma}\) and \(\widehat{\Lambda}\) according to \cite{Ribes-Zalesskii:ProfiniteGroups}*{Proposition~3.2.2, p.\,84}.  In addition, we have \(\overline{H}^{\widehat{\Gamma}} \cong \widehat{H} \cong \overline{H}^{\widehat{\Lambda}}\) because the profinite topologies on \(\Gamma\) and \(\Lambda\) evidently induce the full profinite topology on any finite index subgroup, compare~\cite{Reid:profinite-properties}*{Corollary~4.7\,(1)}.
      \end{proof}

      \subsection{Consequences of profinite commensurability} \label{subsection:consequences}

      As we just saw, the commensurability class of a higher rank \(S\)-arithmetic group \(\Gamma\) can be identified with an ambient class \([k, \mathbf{G}, S]\).  Aka's main result~\cite{Aka:Arithmetic}*{Theorem~3}, translated to our setting, says that the set of all higher rank \(S\)-arithmetic groups \(\Lambda\) with \(\widehat{\Lambda} \cong \widehat{\Gamma}\) lies in a finite union of ambient classes.  The key proposition from which the theorem is concluded will also be the basis of our further investigations.  To state it, let \(\mathfrak{g}_p\) be the Lie algebra of the \(p\)-adic Lie group
  \[ \mathbf{G}_p = \prod_{v \mid p, v \notin S} \mathbf{G}(k_v). \]
  We call \(\mathfrak{g}_p\) the \emph{\(p\)-algebra} of \(\mathbf{G}\) (or of \(\Gamma\)).  Clearly, it is well-defined up to \(\Q_p\)-isomorphism. 
  
   \begin{proposition}[\cite{Aka:Arithmetic}*{Proposition~14}] \label{prop:completiondeterminespalgebras}
     Profinitely commensurable higher rank \(S\)-arithmetic groups have isomorphic \(p\)-algebras.
   \end{proposition}

    Actually, the referenced result assumes the groups are profinitely isomorphic.  But since \(\Gamma\) is residually finite, every open subgroup \(U \le \widehat{\Gamma}\) is given by the closure \(U = \overline{\Gamma_0} \le \widehat{\Gamma}\) of a unique finite index subgroups \(\Gamma_0 \le \Gamma\) and \(\overline{\Gamma_0} \cong \widehat{\Gamma_0}\).  So the given formulation is immediate. One sentence on the proof of~\cite{Aka:Arithmetic}*{Proposition~14}: apply Margulis superrigidity to show that \(\mathfrak{g}_p\) is isomorphic to the Lie algebra of a maximal \(p\)-adic analytic quotient of \(\widehat{\Gamma}\).  To prepare the proof of Theorem~\ref{thm:profcomm-implies-arith-equiv}, it is helpful to recall from~\cite{Aka:Arithmetic} how to extract information on \(k\), \(\mathbf{G}\) and \(S\) from the algebras \(\mathfrak{g}_p\).  One first verifies that the simple ideals of \(\mathfrak{g}_p\) are precisely the Lie algebras \(\lie_{\Q_p} \mathbf{G}(k_v)\) of the factors \(\mathbf{G}(k_v)\) for \(v \mid p\) and \(v \notin S\).  Hence the number of simple ideals of \(\mathfrak{g}_p\) is maximal if and only if \(p\) is split and \emph{\(S\)-unrelated}, meaning \(p\) does not lie below any place in \(S\).  This notion is well-defined because \(\sigma \in \Gal(\overline{\Q}/\Q)\) permutes the places of the Galois closure \(N\) of \(k\) lying over \(p\) so that \(S\) and \(\prescript{\sigma}{}{\!S}\) have the same unrelated primes.  Primes that are split and \(S\)-unrelated always exist because a prime \(p\) splits in \(k\) if and only if it splits in \(N\)~\cite{Neukirch:Zahlentheorie}*{Aufgabe~I.9.4, p.\,62}, hence infinitely many primes \(p\) are split in \(k\) by Chebotarev's density theorem.  Thus the maximal number of simple ideals occurring among the \(\mathfrak{g}_p\) equals \([k : \Q]\).  For every prime \(p\), the dimension of \(\mathfrak{g}_p\) is given by
   \[ \dim_{\Q_p} \mathfrak{g}_p = \dim \mathbf{G} \cdot \sum_{v \mid p, v \notin S} [k_v : \Q_p]. \]
   Here \(\dim \mathbf{G} = \dim_\C \mathbf{G}(\C)\) is the absolute dimension of \(\mathbf{G}\).  Hence \(\dim_{\Q_p} \mathfrak{g}_p\) takes on the maximal value \(\dim \mathbf{G} \cdot [k : \Q]\) if and only if \(p\) is \(S\)-unrelated.  This discussion allows the following conclusion.

  \begin{corollary}[\cite{Aka:Arithmetic}*{Corollary~16}]
    The numbers \([k : \Q]\) and \(\dim \mathbf{G}\) as well as the set of all \(S\)-unrelated primes depend only on the profinite commensurability class of a higher rank \(S\)-arithmetic group \(\Gamma\).
  \end{corollary}

  The observation that \(p\) is split and \(S\)-unrelated if and only if \(\mathfrak{g}_p\) has maximal number of ideals also shows that \(\widehat{\Gamma}\) identifies almost all rational primes that split in \(k\).  As mentioned above, a rational prime splits in \(k\) if and only if it splits in the Galois closure \(N\).  Since a Galois extension is uniquely determined by any finitely complemented subset of the set of split primes~\cite{Neukirch:Zahlentheorie}*{Satz~13.9, p.\,572}, the profinite commensurability class of \(\widehat{\Gamma}\) also determines the Galois closure \(N\) of \(k\)~\cite{Aka:Arithmetic}*{Corollary~16\,(1)}.  In view of Theorem~\ref{thm:consequences-arith-equiv}\,\eqref{item:same-galois-closure}, we now prove a sharpening of this result.

  \begin{proof}[Proof of Theorem~\ref{thm:profcomm-implies-arith-equiv}]
    We have a 1-1 correspondence from the set \(\{ \mathfrak{h}_1, \ldots, \mathfrak{h}_r \}\) of simple ideals of \(\mathfrak{g}_p\) to the set of places above \(p\) and outside \(S\) which sends an ideal \(\mathfrak{h}_i\) to a place \(v_i \mid p\) satisfying \(\mathfrak{h}_i \cong_{\Q_p} \lie_{\Q_p} \mathbf{G}(k_{v_i})\).  The place \(v_i\) determines the dimension of \(\mathfrak{h}_i\) by
    \[ \dim_{\Q_p} \mathfrak{h}_i = \dim_{\Q_p} \lie_{\Q_p} \mathbf{G}(k_{v_i}) = \dim \mathbf{G} \cdot [k_{v_i} \colon \Q_p] = \dim \mathbf{G} \cdot e_{v_i} \cdot f_{v_i} \]
    where \(e_{v_i}\) and \(f_{v_i}\) denote ramification index and residue degree of \(v_i\).  All but finitely many prime numbers \(p\) are both unramified in \(k\) and \(S\)-unrelated.  Hence in almost all cases the unordered tuple
    \[ A_p = \left\{ \frac{\dim_{\Q_p} \mathfrak{h_1}}{\dim \mathbf{G}}, \ldots , \frac{\dim_{\Q_p} \mathfrak{h_r}}{\dim \mathbf{G}} \right\} \]
    gives the decomposition type of \(p\) in \(k\).  
  \end{proof}

  Next we show that any pair of arithmetically equivalent fields leads to examples of profinitely commensurable higher rank \(S\)-arithmetic groups.
  
  \begin{proposition} \label{prop:arith-equiv-fields-give-profinitely-iso-groups}
    For arithmetically equivalent fields \(k\) and \(l\), let \(S\) and \(T\) contain precisely all infinite places and all places over ramified primes.  Then \(\mathbf{SL_3}(\mathcal{O}_{k, S})\) and \(\mathbf{SL_3}(\mathcal{O}_{l, T})\) are profinitely commensurable.
  \end{proposition}

  The proof will additionally reveal that the groups are profinitely isomorphic whenever \(k\) and \(l\) are not totally imaginary.
  
  \begin{proof}
    Since \(k\) and \(l\) are arithmetically equivalent, Theorem~\ref{thm:characterizations-arith-equiv}\,\eqref{item:all-decomp-types-equal} says that for each rational prime \(p\) there exists a residue degree preserving bijection \(\varphi_p
    \) from the places \(v \mid p\) of \(k\) to the places \(w \mid p\) of \(l\).  Whenever \(p\) is unramified in~\(k\) (equivalently in~\(l\), Theorem~\ref{thm:consequences-arith-equiv}\,\eqref{item:same-discriminant}), we have \(k_v \cong l_{\varphi_p(v)}\) for all \(v \mid p\).  Indeed, an unramified extension of \(\Q_p\) is uniquely determined by the residue field \cite{Lang:AlgebraicNumberTheory}*{Proposition~II.\S4.9, p.\,49} which in both cases is the cyclic extension of \(\Z / p\) with degree \(f_v\).   We thus also have \(\mathcal{O}_v \cong \mathcal{O}_{\varphi_p(v)}\) for the valuation rings.  Let us now first assume that \(k\) (equivalently~\(l\), Theorem~\ref{thm:consequences-arith-equiv}\,\eqref{item:same-signatures}) is not totally imaginary.  Then according to the discussion in Section~\ref{subsection:csp}, the congruence kernels \(C(k, \mathbf{SL_3}, S)\) and \(C(l, \mathbf{SL_3}, T)\) are trivial.  Therefore the isomorphisms described in Sections~\ref{subsection:strongapproximation} and~\ref{subsection:csp} together with the bijections \(\varphi_p\) assemble to an isomorphism
    \[ \widehat{\mathbf{SL_3}(\mathcal{O}_{k,S})} \cong \prod_{v \notin S} \mathbf{SL_3}(\mathcal{O}_v) \cong \prod_{w \notin T} \mathbf{SL_3}(\mathcal{O}_w) \cong \widehat{\mathbf{SL_3}(\mathcal{O}_{l,T})}. \]
    If \(k\) is totally imaginary, we explained in Section~\ref{subsection:csp} that we have a short exact sequence
    \[ 1 \longrightarrow \mu_k \longrightarrow \widehat{\mathbf{SL_3}(\mathcal{O}_{k,S})} \xrightarrow{ \ \pi_k \ } \prod_{v \notin S} \mathbf{SL_3}(\mathcal{O}_v) \longrightarrow 1 \]
    and similarly for \(l\) and \(T\).  In a profinite group, the intersection of all open normal subgroups is trivial.  Thus we can find an open, hence finite index subgroup \(U \subset \widehat{\mathbf{SL_3}(\mathcal{O}_{k,S})}\) which intersects the finite group \(\mu_k\) trivially.  This has the effect that \(\pi_k\) embeds \(U\) into \(\prod_{v \notin S} \mathbf{SL_3}(\mathcal{O}_v)\) as a finite index subgroup by surjectivity.  The group \(\pi_k(U)\) is compact, thus closed because profinite groups are Hausdorff.  Therefore \(\pi_k(U)\), being closed and of finite index, is open in \(\prod_{v \notin S} \mathbf{SL_3}(\mathcal{O}_v)\).  Repeating the same construction for \(l\) and \(T\), we obtain an open subgroup \(\pi_l(V)\) in \(\prod_{w \notin T} \mathbf{SL_3}(\mathcal{O}_w)\).  Identifying the two products by the bijections \(\varphi_p\) as above, we can form the intersection \(\pi_k(U) \cap \pi_l(V)\) which is still an open subgroup.  Thus the preimages of this intersection under the restrictions \({\pi_k}_{|U}\) and \({\pi_l}_{|V}\) are isomorphic open subgroups of \(\widehat{\mathbf{SL_3}(\mathcal{O}_{k,S})}\) and \(\widehat{\mathbf{SL_3}(\mathcal{O}_{l,T})}\).
  \end{proof}

  Theorem~\ref{thm:profcomm-implies-arith-equiv} implies that the profinite commensurability class determines the set of places \(S\) of an \(S\)-arithmetic group, at least above unramified primes:

  \begin{proof}[Proof of Theorem~\ref{thm:profcomm-implies-similar-s}]
    By Theorem~\ref{thm:profcomm-implies-arith-equiv} and Theorem~\ref{thm:characterizations-arith-equiv}\,\eqref{item:all-decomp-types-equal}, all decomposition types of all rational primes \(p\) are equal in \(k\) and \(l\).  For an unramified prime \(p\), the dimensions of the simple ideals in \(\mathfrak{g}_p\) yield the residue degrees of the places lying over \(p\) and outside \(S\) as in the proof of Theorem~\ref{thm:profcomm-implies-arith-equiv}.  Thus, also the residue degrees of the places in \(S_p\) are determined by \(\mathfrak{g}_p\).
  \end{proof}

  The next proposition says that Theorem~\ref{thm:profcomm-implies-similar-s} cannot be improved, except possibly by a statement about places over ramified primes.
  
  \begin{proposition} \label{prop:specifiedbijection}
    Let \(k\) and \(l\) be arithmetically equivalent fields.  For finitely many unramified \(p\), let \(S_p\) and \(T_p\) be sets of places in \(k\) and \(l\) above \(p\) such that there are residue degree preserving bijections \(S_p \rightarrow T_p\).  If \(S\) and \(T\) consist of all the sets \(S_p\) and \(T_p\), all infinite places, and all places over ramified primes, then \(\mathbf{SL_3}(\mathcal{O}_{k,S})\) and \(\mathbf{SL_3}(\mathcal{O}_{l, T})\) are profinitely commensurable.
  \end{proposition}

  \begin{proof}
    By default, let us set \(S_p = T_p = \emptyset\) for the remaining \(p\) and let us denote the discriminant of \(k\) by \(\Delta_k\).  Assume first that \(k\) is not totally imaginary.  Incorporating Theorem~\ref{thm:consequences-arith-equiv}\,\eqref{item:same-discriminant}, we have similarly as in Proposition~\ref{prop:arith-equiv-fields-give-profinitely-iso-groups}
    \[ \widehat{\mathbf{SL_3}(\mathcal{O}_{k,S})} \cong \prod_{\substack{p \nmid \infty\\ p \nmid \Delta_k}} \prod_{v \notin S_p} \mathbf{SL_3}(\mathcal{O}_v) \cong \prod_{\substack{p \nmid \infty\\ p \nmid \Delta_l}} \prod_{w \notin T_p} \mathbf{SL_3}(\mathcal{O}_w) \cong \widehat{\mathbf{SL_3}(\mathcal{O}_{l,T})}. \]
    If \(k\) is totally imaginary, the same construction as in Proposition~\ref{prop:arith-equiv-fields-give-profinitely-iso-groups} yields that the groups are profinitely commensurable.
  \end{proof}

  Both in Proposition~\ref{prop:arith-equiv-fields-give-profinitely-iso-groups} and in Proposition~\ref{prop:specifiedbijection} we used the group \(\mathbf{G} = \mathbf{SL_3}\) for simplicity but any other higher rank \(\mathbb{Q}\)-isotropic group \(\mathbf{G}\) would work equally fine.

  \subsection{Profiniteness of group invariants} \label{subsection:profiniteinvariants}  We construct an example of two profinitely isomorphic \(S\)-arithmetic groups with the same underlying \(\Q\)-group but different surrounding Lie groups.
  
    \begin{proof}[Proof of Proposition~\ref{prop:ramified-primes-in-s}]
    The fields \(k=\Q(\sqrt[8]{97})\) and \(l = \Q(\sqrt[8]{16\cdot97})\) are arithmetically equivalent by our discussion in Example~\ref{example:arith-equiv-fields}.  The common discriminant of \(k\) and \(l\) equals \(-2^{10} 97^7\), thus \(2\) and \(97\) are the only ramified primes.  By Proposition~\ref{prop:arith-equiv-fields-give-profinitely-iso-groups}, they define a pair of profinitely isomorphic \(S\)-arithmetic groups \(\mathbf{SL_3}(\mathcal{O}_{k,S})\) and \(\mathbf{SL_3}(\mathcal{O}_{l,T})\).  According to Perlis~\cite{Perlis:OnTheEquation}*{p.\,351}, the prime \(2\) decomposes into four different places both in \(k\) and \(l\) with ramification indices \(1,1,2,4\) and \(2,2,2,2\), respectively.  The prime \(97\) is totally ramified in \(k\) and \(l\).  Let \(k_{97}\) and \(l_{97}\) be the completions of \(k\) and \(l\) with respect to the unique place lying over \(97\).  Then \(\mathbf{SL_3}(\mathcal{O}_{k,S})\) and \(\mathbf{SL_3}(\mathcal{O}_{l,T})\) are lattices in the Lie groups
    \begin{align*}
      (\textup{SL}_3)_S &= \prod_{v \in S} \mathbf{SL_3}(k_v) = \textup{SL}_3(\infty) \times \mathbf{SL_3}(k_{97}) \times \prod_{v \mid 2} \mathbf{SL_3}(k_v) \text{ and}\\
      (\textup{SL}_3)_T &= \prod_{w \in T} \mathbf{SL_3}(l_w) = \textup{SL}_3(\infty) \times \mathbf{SL_3}(l_{97}) \times \prod_{w \mid 2} \mathbf{SL_3}(l_w),
    \end{align*}
 respectively, where \(\textup{SL}_3(\infty) = \mathbf{SL_3}(\R) \times \mathbf{SL_3}(\R) \times \mathbf{SL_3}(\C) \times \mathbf{SL_3}(\C) \times \mathbf{SL_3}(\C)\).  Both groups are products of a real Lie group, a 97-adic Lie group and a 2-adic Lie group.  In \((\textup{SL}_3)_S\) the Lie algebra of the \(2\)-adic factor has simple ideals of dimensions \(8,8,16,32\) while in \((\textup{SL}_3)_T\) the Lie algebra of the \(2\)-adic factor has four simple ideals of dimension \(16\) each.  Thus \((\textup{SL}_3)_S\) and \((\textup{SL}_3)_T\) cannot be isomorphic.
    \end{proof}

    We will now see that examples of this type can be ruled out by assuming that the set \(S\) is unramified.
  
  \begin{proof}[Proof of Proposition~\ref{prop:arith-group-determines-lie-group}]
    If \([k, \mathbf{G}, S]\) and \([l, \mathbf{G}, T]\) are profinitely commensurable, then \(k\) and \(l\) are arithmetically equivalent by Theorem~\ref{thm:profcomm-implies-arith-equiv}.  By Theorem~\ref{thm:consequences-arith-equiv}\,\eqref{item:same-signatures}, \(k\) and \(l\) have the same number of real and complex infinite places.  By Theorem~\ref{thm:profcomm-implies-similar-s}, there are residue degree preserving bijections \(\varphi_p \colon S_p \rightarrow T_p\) for all unramified primes \(p\).  As we already discussed in the proof of Proposition~\ref{prop:arith-equiv-fields-give-profinitely-iso-groups}, we thus obtain field isomorphisms \(k_v \cong l_{\varphi_p(v)}\) for all \(v \in S_p\) with unramified~\(p\).  Since by assumption \(S_p = T_p = \emptyset\) whenever \(p\) is ramified, we thus obtain
    \[ G_S = \prod_{v \mid \infty} \mathbf{G}(k_v) \times \prod_p \prod_{v \in S_p} \mathbf{G}(k_v) \cong \prod_{w \mid \infty} \mathbf{G}(l_w) \times \prod_p \prod_{w \in T_p} \mathbf{G}(l_w) = G_T. \qedhere \]
  \end{proof}

  We now prove Theorem~\ref{thm:l2bettiprofinite} which gives a consequence of these observations for the \(\ell^2\)-cohomology of \(S\)-arithmetic groups.

  \begin{proof}[Proof of Theorem~\ref{thm:l2bettiprofinite}]
    Kyed--Petersen--Vaes \cite{Kyed-Petersen-Vaes:LocallyCompact}*{Theorem~B} showed that the \emph{\(\ell^2\)-Betti numbers} \(b^{(2)}_n(\Gamma)\) of a lattice \(\Gamma \le G\) in a locally compact group \(G\) with Haar measure \(\mu\) are given by \(b^{(2)}_n(\Gamma) = b^{(2)}_n(G, \mu) \cdot \mu(\Gamma \backslash G)\).  Here \(b^{(2)}_n(G, \mu)\) is the \emph{\(\ell^2\)-Betti number of the locally compact group \(G\)} as defined by Petersen~\cite{Petersen:PhD} and \(\mu(\Gamma \backslash G) > 0\) is the induced \(G\)-invariant measure on the quotient space.  The assertion of the theorem is therefore immediate from Proposition~\ref{prop:arith-group-determines-lie-group}.
\end{proof}

  Here are some explanations on the example below Theorem~\ref{thm:l2bettiprofinite} \label{page:localfields}.  To begin with, one can establish that the polynomials in~\eqref{equation:mantilla} define arithmetically equivalent number fields by means of Perlis' criterion~\cite{Perlis:Remark} which states that \(k\) and \(l\) of prime degree \(p\) are arithmetically equivalent if and only if \([k\!\cdot\!l : \Q] < p^2\).  The arithmetic equivalence of \(k\) and \(l\) already implies that \(k \otimes_\Q \Q_p \cong_{\Q_p} l \otimes_\Q \Q_p\) for all unramified \(p\).  To see that \(k\) and \(l\) are adelically equivalent, it thus remains to analyze the completions at ramified primes.  Since the common discriminant of \(k\) and \(l\) is \(2^6\,13^4\,191^2\), we only need to compare \(k \otimes_\Q \Q_p\) to \(l \otimes_\Q \Q_p\) for \(p = 2, 13, 191\).  With the \emph{Database of Local Fields} due to Jones--Roberts~\cite{Jones-Roberts:Database}, it is conveniently verified that these \(\Q_p\)-algebras are likewise isomorphic.  Interestingly, the two fields \(k\) and \(l\) have non-isomorphic integral trace forms as can be seen by the methods of~\cite{Mantilla:TraceForms}.  Of course this implies in particular that they are not isomorphic.

  Along the lines of the proof of Proposition~\ref{prop:arith-equiv-fields-give-profinitely-iso-groups}, we see that the groups \(\mathbf{Spin(3,2)}(\mathcal{O}_{k, S})\) and \(\mathbf{Spin(3,2)}(\mathcal{O}_{l, T})\) are profinitely commensurable, and in fact profinitely isomorphic.  The assumption that \(S\) contains no places over ramified primes effects that the full algebra \(k \otimes_\Q \Q_p\) is contained in \(\mathbb{A}_{k,S}\) for \(p = 2, 13, 191\) and similarly for \(l\) and \(T\) so that \(\prod_{v \notin S} \mathbf{Spin(3,2)}(\mathcal{O}_v) \cong \prod_{w \notin T} \mathbf{Spin(3,2)}(\mathcal{O}_w)\) by the adelic equivalence of \(k\) and \(l\).

  Finally, we check for which \(n\) the \(\ell^2\)-Betti number \(b^{(2)}_n((\textup{Spin}(3,2))_S)\) is nonzero.  For this question one can clearly suppress the Haar measures from the notation.  The real semisimple Lie group
  \[ G = \prod_{v \mid \infty}  \mathbf{Spin(3,2)}(\R) = \left(\mathbf{Spin(3,2)}(\R)\right)^7 \]
  has maximal compact subgroup \(K=(\textup{Spin}(3) \times \textup{Spin}(2))^7\) with Lie algebras \(\mathfrak{g}\) and \(\mathfrak{k}\) for which we have \(\rank_\C \,\mathfrak{g} \otimes_\R \C = \rank_\C \,\mathfrak{k} \otimes_\R \C\).  This shows that \(G\) has discrete series representations from which it follows that \(b^{(2)}_n(G) > 0\) if and only if \(n = 21\).  In our terms, this can be concluded from~\cite{Petersen-Valette:Plancherel}*{Corollary~11} though the calculation is due to Borel~\cite{Borel:NegativelyCurved}.  The degree \(n = 21\) is precisely the middle dimension of the associated symmetric space.  For a finite place \(v\) of \(k\), the totally disconnected locally compact group \(\mathbf{Spin(3,2)}(k_v)\) can only have a positive \(\ell^2\)-Betti number in degree two which is the \(k_v\)-rank of \(\mathbf{Spin(3,2)}\)~\cite{Petersen-Valette:Plancherel}*{Corollary~13}.  In fact, in this case we have \(b^{(2)}_2(\mathbf{Spin(3,2)}(k_v)) = \nu(\textup{St})\) where \(\nu\) is the Plancherel measure on the unitary dual \(\widehat{G}\) and \(\textup{St}\) is the \emph{Steinberg module}.  The latter is a discrete series representation and those representations form precisely the atoms of the Plancherel measure so that indeed \(b^{(2)}_2(\mathbf{Spin(3,2)}(k_v)) > 0\).  Petersen's K\"unneth formulas~\cite{Petersen:PhD}*{Theorems~6.5 and~6.7} for products of locally compact groups thus show that all \(\ell^2\)-Betti numbers of \((\textup{Spin}(3,2))_S\) vanish, except in degree \(n = 21 + 2(|S|-7)\) when
  \[ b^{(2)}_{7 + 2|S|} ((\textup{Spin}(3,2))_S) = b^{(2)}_{21}(G) \prod_{\substack{v \in S\\ v \nmid \infty}} b^{(2)}_2(\mathbf{Spin(3,2)}(k_v)) > 0. \]
  We remark that, strictly speaking, \cite{Petersen-Valette:Plancherel}*{Corollary~13} contains the assumption that the residue field of \(k_v\) should be large (of order \(> \frac{1724^2}{25}\)).  But A.\,Valette informs us that it should be possible to drop this assumption by transferring methods of Borel--Wallach from finite-dimensional to infinite-dimensional representations although this has not been written up yet.

  We conclude easily from Theorem~\ref{thm:l2bettiprofinite} that the Euler characteristic has the same sign for all the profinitely commensurable groups it deals with.

  \begin{proof}[Proof of Corollary~\ref{cor:signofeuler}.]
    Just as for ordinary Betti numbers, we have the alternating sum formula \(\chi(\Gamma) = \sum_{n \ge 0} (-1)^n b^{(2)}_n(\Gamma)\), see \cite{Lueck:L2-Invariants}*{Theorem~1.35.(2), p.\,37} or \cite{Kammeyer:IntroL2}*{Theorem~3.19}.  But as we discussed above, \(S\)-arithmetic subgroups of simple algebraic groups have a nonzero \(\ell^2\)-Betti number in at most one degree.  Theorem~\ref{thm:l2bettiprofinite} says that this degree is the same for the groups under consideration.
  \end{proof}

  Finally, we explain why we need no assumptions on \(S\) and \(T\) in Theorem~\ref{thm:l2bettiprofinite} and Corollary~\ref{cor:signofeuler} if \(\mathbf{G}\) splits over \(k\) and \(l\).  By Theorem~\ref{thm:profcomm-implies-arith-equiv} and Theorem~\ref{thm:consequences-arith-equiv}\,\eqref{item:same-signatures}, the real and complex factors in the Lie group products \(G_S\) and \(G_T\) are isomorphic.  The \(p\)-adic factors in \(G_S\) and \(G_T\) may differ, as we saw in Proposition~\ref{prop:ramified-primes-in-s}, but their number is the same.  Indeed, for each rational prime \(p\) we know that the number of places over \(p\) is equal in \(k\) and \(l\) by Theorem~\ref{thm:profcomm-implies-arith-equiv} and Theorem~\ref{thm:characterizations-arith-equiv}\,\eqref{item:all-decomp-types-equal}.  Also, the number of places in \(k\) above \(p\) and outside \(S\) is equal to the number of places in \(l\) above \(p\) and outside \(T\) because this number equals the number of simple ideals in \(\mathfrak{g}_p\).  Put together, this implies \(|S_p| = |T_p|\) for all \(p\), thus \(|S| = |T|\).  Since we assume that \(\mathbf{G}\) splits over \(k\) and \(l\), we have \(\rank_{k_v} \mathbf{G} = \rank_{l_w} \mathbf{G} = \rank_\C \mathbf{G} = r\) for all places \(v\) in \(k\) and \(w\) in \(l\).  Thus by the K\"unneth formula, the only nonzero \(\ell^2\)-Betti number of the totally disconnected factor in \(G_S\) and \(G_T\) sits in the same degree \(r(|S|-|\{ v \mid \infty \}|)\).  Again by the K\"unneth formula, \(b^{(2)}_n(G_S) = 0\) if and only if \(b^{(2)}_n(G_T) = 0\) and this implies the result.  Note that instead of requiring that \(\mathbf{G}\) be split over \(k\) and \(l\), it would actually have sufficed to assume that \(\mathbf{G}\) splits over all local fields corresponding to the finite places in \(S\) and \(T\).

\end{document}